\documentclass[12 pt]{amsart}
\usepackage{amssymb,latexsym,amsmath,amscd,amsthm,graphicx, color}
\usepackage[all]{xy}
\usepackage{pgf,tikz}
\usepackage{mathrsfs}
\usepackage{cite}

\usetikzlibrary{arrows}
\definecolor{uuuuuu}{rgb}{0.26666666666666666,0.26666666666666666,0.26666666666666666}
\definecolor{xdxdff}{rgb}{0.49019607843137253,0.49019607843137253,1.}
\definecolor{ffqqqq}{rgb}{1.,0.,0.}

\definecolor{uuuuuu}{rgb}{0.26666666666666666,0.26666666666666666,0.26666666666666666}
\definecolor{qqwuqq}{rgb}{0.,0.39215686274509803,0.}
\definecolor{zzttqq}{rgb}{0.6,0.2,0.}
\definecolor{xdxdff}{rgb}{0.49019607843137253,0.49019607843137253,1.}
\definecolor{qqqqff}{rgb}{0.,0.,1.}
\definecolor{cqcqcq}{rgb}{0.7529411764705882,0.7529411764705882,0.7529411764705882}

\definecolor{uuuuuu}{rgb}{0.26666666666666666,0.26666666666666666,0.26666666666666666}
\definecolor{qqwuqq}{rgb}{0.,0.39215686274509803,0.}
\definecolor{zzttqq}{rgb}{0.6,0.2,0.}
\definecolor{xdxdff}{rgb}{0.49019607843137253,0.49019607843137253,1.}
\definecolor{qqqqff}{rgb}{0.,0.,1.}
\definecolor{cqcqcq}{rgb}{0.7529411764705882,0.7529411764705882,0.7529411764705882}

\theoremstyle{plain}

\newtheorem{theorem}[subsection]{Theorem}

\newtheorem{lemma}[subsection]{Lemma}

\newtheorem{prop}[subsection]{Proposition}

\newtheorem{defi}[subsection]{Definition}

\theoremstyle{definition}

\newtheorem{cor}[subsection]{Corollary}

\newtheorem{remark}[subsection]{Remark}

\newcommand{\uu}{\cup}
\newcommand{\ii}{\cap}
\newcommand{\UU}{\bigcup}

\newcommand{\ci}{\subseteq}
\newcommand{\sci}{\subset}
\newcommand{\es}{\emptyset}
\newcommand{\set}[1]{\{#1\}}

\newcommand{\ga}{\alpha}
\newcommand{\gb}{\beta}

\newcommand{\gd}{\delta}
\renewcommand{\gg}{\gamma}

\newcommand{\gk}{\kappa}

\newcommand{\go}{\omega}

\newcommand{\gs}{\sigma}

\newcommand{\tit}{\textit}

\newcommand{\D}[1]{\mathbb{#1}}

\newcommand{\te}{\text}

\begin{document}

To appear, Qualitative Theory of Dynamical Systems
\title{Optimal quantization for some triadic uniform Cantor distributions with exact bounds}
\author{Mrinal Kanti Roychowdhury}
\address{School of Mathematical and Statistical Sciences\\
University of Texas Rio Grande Valley\\
1201 West University Drive\\
Edinburg, TX 78539-2999, USA.}
\email{mrinal.roychowdhury@utrgv.edu}

\thanks{ }

\date{}
\maketitle

\pagestyle{myheadings}\markboth{Mrinal Kanti Roychowdhury}{Optimal quantization for some triadic uniform Cantor distributions with exact bounds}

\begin{abstract}   Let $\{S_j : 1\leq j\leq 3\}$ be a set of three contractive similarity mappings such that $S_j(x)=rx+\frac {j-1}{2}(1-r)$ for all $x\in \mathbb R$, and $1\leq j\leq 3$, where $0<r<\frac 1 3$. Let  $P=\sum_{j=1}^3 \frac 13 P\circ S_j^{-1}$. Then, $P$ is a unique Borel probability measure on $\mathbb R$ such that $P$ has support the Cantor set generated by the similarity mappings $S_j$ for $1\leq j\leq 3$. Let $r_0=0.1622776602$, and $r_1=0.2317626315$ (which are ten digit rational approximations of two real numbers).  In this paper, for $0<r\leq r_0$, we give a general formula to determine the optimal sets of $n$-means and the $n$th quantization errors for the triadic uniform Cantor distribution $P$ for all positive integers $n\geq 2$. Previously, Roychowdhury gave an exact formula to determine the optimal sets of $n$-means and the $n$th quantization errors for the standard triadic Cantor distribution, i.e., when $r=\frac 15$. In this paper, we further show that $r=r_0$ is the greatest lower bound, and $r=r_1$ is the least upper bound of the range of $r$-values to which Roychowdhury formula extends. In addition, we show that for $0<r\leq r_1$ the quantization coefficient does not exist though the quantization dimension exists.
\end{abstract}

\section{Introduction}

Let $P$ be a Borel probability measure on $\D R^d$, where $d\geq 1$. For a finite set $\ga \sci \D R^d$, write
\[V(P; \ga)= \int \min_{a\in\alpha} \|x-a\|^2 dP(x), \te{ and } V_n:=V_n(P)=\inf \Big\{V(P; \ga) : \ga \subset \mathbb R^d, \text{ card}(\ga) \leq n \Big\},\]
where $\|\cdot\|$ represents the Euclidean norm on $\D R^d$. Then, $V(P;\ga)$ is called the \tit{cost} or \tit{distortion error} for $P$ with respect to the set $\ga$, and $V_n$ is called the $n$th quantization error for $P$ with respect to the squared Euclidean distance. A set $\ga\sci \D R^d$ is called an \tit{optimal set of $n$-means} for $P$ if $V_n(P)=V(P; \ga)$. It is well-known that for a continuous Borel probability measure an optimal set of $n$-means contains exactly $n$-elements (see \cite{GL1}). To see some work in the direction of optimal sets of $n$-means, one is referred to \cite{DR, GL2, RR}. For theoretical results in quantization we refer to
\cite{GL1, GL3, GL4, GL5, P}, and for its promising application see \cite{P1, P2}.
For a finite set $\ga \sci \D R^d$ and $a\in \ga$, by $M(a|\ga)$ we denote the set of all elements in $\D R^d$ which are nearest to $a$ among all the elements in $\ga$, i.e.,
\[M(a|\ga)=\set{x \in \D R^d : \|x-a\|=\min_{b \in \ga}\|x-b\|}.\]
$M(a|\ga)$ is called the \tit{Voronoi region} generated by $a\in \ga$. On the other hand, the set $\set{M(a|\ga) : a \in \ga}$ is called the \tit{Voronoi diagram} or \tit{Voronoi tessellation} of $\D R^d$ with respect to the set $\ga$.
\begin{defi} \label{defi000}
A set $\ga\sci \D R^d$ is called a \tit{centroidal Voronoi tessellation} (CVT) with respect to a probability distribution $P$ on $\D R^d$,  if it satisfies the following two conditions:

$(i)$ $P(M(a|\ga)\ii M(b|\ga))=0$ for $a, b\in \ga$, and $a \neq b$;

$(ii)$ $E(X : X \in M(a|\ga))=a$ for all $a\in \ga$,

where $X$ is a random variable with distribution $P$, and $E(X : X \in M(a|\ga))$ represents the conditional expectation of the random variable $X$ given that $X$ takes values in $M(a|\ga)$.
\end{defi}
A Borel measurable partition $\set{A_a : a\in \ga}$ is called a \tit{Voronoi partition} of $\D R^d$ with respect to the probability distribution $P$, if $P$-almost surely  $A_a\sci M(a|\ga)$ for all $a\in \ga$.
Let us now state the following proposition (see \cite{GG, GL1}).

\begin{prop} \label{prop0}
Let $\ga$ be an optimal set of $n$-means, $a \in \alpha$, and $M(a|\ga)$ be the Voronoi region generated by $a\in \ga$, i.e.,
$M(a|\ga)=\{x \in \mathbb R^d : \|x-a\|=\min_{b \in \alpha} \|x-b\|\}.$
Then, for every $a \in\alpha$,
$(i)$ $P(M(a|\ga))>0$, $(ii)$ $ P(\partial M(a|\ga))=0$, $(iii)$ $a=E(X : X \in M(a|\ga))$.
\end{prop}
The number
$ D(P):=\mathop{\lim}\limits_{n\to \infty}  \frac{2\log n}{-\log V_n(P)},$
if it exists, is called the \tit{quantization dimension} of the probability measure $P$. On the other hand, for $s\in (0, +\infty)$, the number $\mathop{\lim}\limits_{n\to \infty} n^{\frac 2 s} V_n(P)$, if it exists, is called the  $s$-dimensional \tit{quantization coefficient} for $P$. To know details about the quantization dimension and the quantization coefficient one is referred to \cite{GL1}.

Let $\set{S_j : 1 \leq j\leq 3}$ be a set of three contractive similarity mappings such that $S_j(x)=r x +\frac{j-1} {2}(1-r)$ for all $x\in \D R$, where $0<r<\frac 13$ and $1\leq j\leq 3$. For any positive integer $n$, if $\gs:=\gs_1\gs_2 \cdots\gs_n \in \{ 1, 2,3\}^n$, then we say that $\gs$ is a word of length $n$. By $\{1, 2, 3\}^*$, we denote the set of all words including the empty word $\es$. The empty word $\es$ has length zero. For $\gs:=\gs_1\gs_2\cdots \gs_n\in \{1, 2, 3\}^n$, by $S_\gs$ it is meant that $S_\gs:=S_{\gs_1}\circ \cdots \circ S_{\gs_n}$, and by $a(\gs)$, we mean $a(\gs):=S_\gs(\frac 12)$. For the empty word $\es$, by $S_\es$ it is meant the identity mapping on $\D R$. For $\gs:=\gs_1\gs_2 \cdots\gs_n \in\{1, 2, 3\}^n$, set $J_\gs:=S_{\gs}([0, 1])$. For the empty word $\es$, write $J:=J_\es=S_\es([0,1])=[0, 1]$. Then, the set $C:=\bigcap_{n\in \mathbb N} \bigcup_{\gs \in \{1, 2, 3\}^n} J_\gs$ is known as the \textit{Cantor set} generated by the mappings $S_j$, and equals the support of the probability measure $P$ given by $P=\sum_{j=1}^3 \frac 1 3 P\circ S_j^{-1}$.
 Notice that $C$ satisfies the invariance equality $C=\mathop{\uu}\limits_{j=1}^3 S_j(C)$  (see \cite{H}). In this paper a Cantor set $C$, which is generated by a set of three contractive similarity mappings, is called a \tit{triadic Cantor set}, and a probability measure $P$ which has support the triadic Cantor set, is called a \tit{triadic Cantor distribution}.
For words $\gb, \gg, \cdots, \gd$ in $\set{1, 2, 3}^\ast$, we write
\begin{equation*} \label{eq45} a(\gb, \gg, \cdots, \gd):=E(X|X\in J_\gb \uu J_\gg \uu \cdots \uu J_\gd)=\frac{1}{P(J_\gb\uu \cdots \uu J_\gd)}\int_{J_\gb\uu \cdots \uu J_\gd} x dP(x),
\end{equation*}
where $X$ is a random variable with probability distribution $P$, and $E(X)$ and $V:=V(X)$ represent the expectation and the variance of the random variable $X$.
Notice that for any $\go\in \set{1,2, 3}^\ast$, the similarity mapping $S_\go$ is an injective mapping on $\D R$; on the other hand,  for any discrete subset $A$ of $\D R$, the set $S_\go(A)$ represents the set of values obtained by applying $S_\go$ to each of the elements in $A$. Let us now give the following two definitions.

\begin{defi}\label{defi00} For $n\in \D N$ with $n\geq 3$ let $\ell(n)$ be the unique natural number with $3^{\ell(n)} \leq n< 3^{\ell(n)+1}$. Write
$\gb_2:=\set{a(1), a(2, 3)}$ and $\gb_3:=\set{a(1), a(2), a(3)}$. For $n\geq 3$, define $\gb_n:=\gb_n(I)$ as follows: \[\gb_n(I)=\left\{\begin{array}{cc}
\set{a(\go): \go \in \set{1, 2, 3}^{\ell(n)}\setminus I}\UU \mathop{\uu}\limits_{\go\in I} S_\go(\gb_2) & \te{ if } 3^{\ell(n)}\leq n\leq 2\cdot 3^{\ell(n)},\\
 \set{S_\go(\gb_2) : \go \in \set{1, 2, 3}^{\ell(n)}\setminus I}\UU \mathop{\uu}\limits_{\go\in I} S_\go(\gb_3) & \te{ if } 2\cdot 3^{\ell(n)}< n< 3^{\ell(n)+1},
\end{array}
\right.\] where $I \sci \set{1, 2, 3}^{\ell(n)}$ is arbitray with $\te{card}(I)=n-3^{\ell(n)}$ if $3^{\ell(n)}\leq n\leq 2\cdot3^{\ell(n)}$; and
$\te{card}(I)=n-2\cdot 3^{\ell(n)}$ if $2\cdot 3^{\ell(n)}< n< 3^{\ell(n)+1}$.
\end{defi}

\begin{defi}\label{defi23} For $n\in \D N$ with $n\geq 3$ let $\ell(n)$ be the unique natural number with $3^{\ell(n)} \leq n< 3^{\ell(n)+1}$. Write
$\gg_2:=\set{a(1, 21), a(22, 23, 3)}$ and $\gg_3:=\set{a(1), a(2), a(3)}$. For $n\geq 3$, define $\gg_n:=\gg_n(I)$ as follows: \[\gg_n(I)=\left\{\begin{array}{cc}
\set{a(\go) : \go \in \set{1, 2, 3}^{\ell(n)}\setminus I}\UU \mathop{\uu}\limits_{\go\in I} S_\go(\gg_2) & \te{ if } 3^{\ell(n)}\leq n\leq 2\cdot 3^{\ell(n)},\\
 \set{S_\go(\gg_2) : \go \in \set{1, 2, 3}^{\ell(n)}\setminus I}\UU \mathop{\uu}\limits_{\go\in I} S_\go(\gg_3) & \te{ if } 2\cdot 3^{\ell(n)}< n< 3^{\ell(n)+1},
\end{array}
\right.\] where $I \sci \set{1, 2, 3}^{\ell(n)}$ is arbitrary with $\te{card}(I)=n-3^{\ell(n)}$ if $3^{\ell(n)}\leq n\leq 2\cdot3^{\ell(n)}$; and
$\te{card}(I)=n-2\cdot 3^{\ell(n)}$ if $2\cdot 3^{\ell(n)}< n< 3^{\ell(n)+1}$.
\end{defi}
\begin{remark}
In the paper there are several decimal numbers, they are rational approximations of some real numbers up to ten decimal places.
\end{remark}
Roychowdhury showed that if $r=\frac 15$, then the sets $\gg_n$ given by Definition~\ref{defi00}, determine the optimal sets of $n$-means for all positive integers $n\geq 2$ (see \cite{R2}).
Proposition~\ref{prop004} implies that $\gg_n$ forms a CVT if $\frac{1}{79} \left(21-2 \sqrt{51}\right)\leq r\leq \frac{1}{41} \left(2 \sqrt{31}-1\right)$, i.e., if $0.08502712839\leq r\leq 0.2472080177$. Thus, we see that the range of $r$ values for which the sets $\gg_n$ form the optimal sets of $n$-means is bounded below by $\frac{1}{79} \left(21-2 \sqrt{51}\right)$, and bounded above by $\frac{1}{41} \left(2 \sqrt{31}-1\right)$. But, the greatest lower bound and the least upper bound of the range of $r$ values for which the sets $\gg_n$ form the optimal sets of $n$-means were not known. In this paper, in Theorem~\ref{Th3} we give an answer of it.
\begin{remark}
Notice that if $r=0$, then $S_1(x)=0$, $S_2(x)=\frac 12$, and $S_3(x)=1$ for all $x\in \D R$, and then the probability measure $P$ becomes a discrete uniform distribution with support $\set{0, \frac 12, 1}$. Because of that in our study we are assuming that the contractive ratios $r$ are positive.
\end{remark}

The arrangement of the paper is as follows: In Section~\ref{sec1}, we give the basic preliminaries. In Section~\ref{sec2}, we show that the sets $\gb_n$ form the optimal sets of $n$-means if $r=\frac 1{25}$.  In Section~\ref{sec3}, we prove the following theorem:
\begin{theorem} \label{Th2} Let $\gg_n:=\gg_n(I)$ be the set for arbitrary $I$ as defined by Definition~\ref{defi23}. Let $r_0, r_1\in (0, \frac 13)$ be the unique real numbers satisfying, respectively, the equations
\begin{align*} -\frac{3 r^5+15 r^4+6 r^3-42 r^2+31 r-13}{240 (r+1)}&=-\frac{3 r^3-3 r^2+r-1}{24 (r+1)},\\
-\frac{3 r^5+15 r^4+6 r^3-42 r^2+31 r-13}{240 (r+1)}&=-\frac{3 r^7+15 r^6+60 r^5+66 r^4+18 r^3-324 r^2+283 r-121}{2184 (r+1)}.
\end{align*}
Then, $r_0=0.1622776602$, and $r_1=0.2317626315$. Then, for all $n\geq 3$, the sets $\gg_n$ form the optimal sets of $n$-means for $r=r_0$ and $r=r_1$.
\end{theorem}
In Theorem~\ref{Th3}, we show that the sets $\gb_n$ form the optimal sets of $n$-means if $0<r\leq r_0$, and the sets $\gg_n$ form the optimal sets of $n$-means if $r_0\leq r\leq r_1$. Thus, Theorem~\ref{Th3} implies the fact that the greatest lower bound, and the least upper bound of $r$ for which the sets $\gg_n$ form the optimal sets of $n$-means are, respectively, given by $r=r_0$ and $r=r_1$. Notice that for $r=r_0$ both the sets $\gb_n$ and $\gg_n$ form the optimal sets of $n$-means for $P$. In addition, in Theorem~\ref{Th4}, we show that the quantization coefficient for $0<r\leq r_1$ does not exist though the quantization dimension exists.

\section{Preliminaries} \label{sec1}

As defined in the previous section, let $S_j$ for $1\leq j\leq 3$ be the contractive similarity mappings on $\D R$ given by $S_j(x)=r x +\frac{j-1} {2}(1-r)$ for all $x\in \D R$, and $1\leq j\leq 3$, where $0<r<\frac 13$.
For $\gs:=\gs_1\gs_2\cdots \gs_k\in \{1, 2, 3\}^k$ and
$\tau:=\tau_1\tau_2\cdots \tau_\ell\in \{1, 2, 3\}^\ell$, by
$\gs\tau:=\gs_1\cdots \gs_k\tau_1\cdots \tau_\ell$ we mean the word obtained from the
concatenation of the words $\gs$ and $\tau$. For $\gs=\gs_1\gs_2 \cdots\gs_n \in \{ 1, 2, 3\}^\ast$, $n\geq 0$, write
$p_\gs:=\frac 1{3^n}$ and $s_\gs:=\frac 1 {r^n}$. Recall that if $C$ is the Cantor set, then $C:=\bigcap_{n\in \mathbb N} \bigcup_{\gs \in \{1, 2, 3\}^n} J_\gs$. For $n\geq 1$, the intervals $J_\gs$, where $\gs\in \set{1, 2,3}^n$, are called the \tit{$n$th level basic intervals} of the Cantor set $C$.

The following two lemmas are well-known and easy to prove (see \cite{GL2, R2}).
\begin{lemma} \label{lemma1}
Let $f : \mathbb R \to \mathbb R^+$ be Borel measurable and $k\in \mathbb N$, and $P$ be the probability measure on $\D R$ given by $P=\sum_{j=1}^3 \frac 1 3P\circ S_j^{-1}$. Then,
\[\int f(x) dP(x)=\sum_{\sigma \in \{1, 2, 3\}^k} \frac 1 {3^k} \int f \circ S_\sigma(x) dP(x).\]
\end{lemma}

\begin{lemma} \label{lemma2} Let $X$ be a random variable with the probability distribution $P$. Then,
\[E(X)=\frac 12 \te{ and }  V:=V(X)=\frac{1-r}{6 (r+1)}, \te{ and } \int (x-x_0)^2 dP(x) =V (X) +(x_0-\frac 12)^2,\]
where $x_0\in \D R$.
\end{lemma}

The following corollary is useful to obtain the distortion errors.
\begin{cor} \label{cor1}
Let $\gs \in \set{1, 2, 3}^k$ for $k\geq 1$, and $x_0 \in \mathbb R$. Then,
\begin{equation*} \label{eq234} \int_{J_\gs} (x-x_0)^2 dP(x) =\frac 1{3^k} \Big(r^{2k} V  +(S_\gs(\frac 12)-x_0)^2\Big).\end{equation*}
\end{cor}

\begin{proof}
By induction, $P=\frac 13 \sum_{j=1}^3 P\circ S_j^{-1}$ implies $P=\sum_{\gs\in\set{1, 2,3}^k} p_\gs P\circ S_\gs^{-1}$. Using this fact, Lemma~\ref{lemma1} and Lemma~\ref{lemma2}, the proof of the corollary follows.
\end{proof}

\begin{prop} \label{prop002}
Let $\gb_n(I)$ be the set given by Definition~\ref{defi00}. Then, $\gb_n(I)$ forms a CVT if $0<r\leq 2-\sqrt 3$, i.e., if $0<r\leq 0.2679491924$. Moreover, if  $3^{\ell(n)} \leq n\leq 2 \cdot 3^{\ell(n)}$, then
\[V(P, \gb_n(I))=\frac 1 {3^{\ell(n)}} \cdot r^{2\ell(n)} \Big ((2\cdot 3^{\ell(n)}-n) V+(n-3^{\ell(n)}) V(P; \gb_2)\Big),\]
and if $2\cdot 3^{\ell(n)} \leq n<3^{\ell(n)+1}$, then
\[V(P, \gb_n(I))=\frac 1 {3^{\ell(n)}} \cdot r^{2\ell(n)} \Big ((3^{\ell(n)+1}-n) V(P; \gb_2)+(n-2\cdot 3^{\ell(n)}) V(P; \gb_3)\Big).\]
\end{prop}
\begin{proof}
By the definition, we have $\gb_2=\set{a(1), a(2, 3)} \te{ and }  \gb_3=\set{a(1), a(2), a(3)}.$
 Recall that $\gb_n:=\gb_n(I)$ is defined for $n\geq 3$, where $I \sci \set{1, 2, 3}^{\ell(n)}$ with $\te{card}(I)=n-3^{\ell(n)}$ if $3^{\ell(n)}\leq n\leq 2\cdot3^{\ell(n)}$; and
$\te{card}(I)=n-2\cdot 3^{\ell(n)}$ if $2\cdot 3^{\ell(n)}< n< 3^{\ell(n)+1}$. Notice that for $n\geq 3$, if $n\neq 3^{\ell(n)}$ or $n\neq 2\cdot 3^{\ell(n)}$, the subset $I$ can be chosen more than one way. This leads to the fact that if $n\neq 3^{\ell(n)}$ or $n\neq 2\cdot 3^{\ell(n)}$, the sets $\gb_n$ can be chosen multiple ways. Let us take
\begin{align*}
 \gb_4 &=\set{a(1), a(2), a(31),  a(32, 33)} \, (\te{by choosing } I=\set{3}),\\
 \gb_5&=\set{a(1), a(21),  a(22, 23),  a(31),  a(32, 33)} \, (\te{by choosing } I=\set{2, 3}),\\
 \gb_6&=\set{a(11), a(12, 13), a(21),  a(22, 23),  a(31),  a(32, 33)} \,(\te{where } I=\set{1, 2, 3}),\\
  \gb_7&=\set{a(11), a(12), a(13), a(21),  a(22, 23),  a(31),  a(32, 33)} \, (\te{by choosing } I=\set{1}).
\end{align*}
Since similarity mappings preserve the ratio of the distances of a point from any other two points, $\gb_n(I)$ will form a CVT if we can show that
$\gb_2,\, \gb_3, \, \gb_4, \, \gb_5, \, \gb_6,\, \gb_7$ form a CVT. Recall that $a(1)=E(X : X \in J_1)$ and $a(2, 3)=E(X : X\in J_2\uu J_3)$, and also recall the Definition~\ref{defi000}. Thus, $\gb_2$ will form a CVT if
\begin{equation} \label{eq90} P(M(a(1)|\gb_2)\ii M(a(2,3)|\gb_2))=0.
\end{equation}
Since the basic intervals in the first level are $J_1:=[S_1(0), S_1(1)]$, $J_2:=[S_2(0), S_2(1)]$, and $J_3:=[S_3(0), S_3(1)]$, the relation~\eqref{eq90} will be true if
\[S_{1}(1)\leq \frac{1}{2} \left(a(1)+ a(2, 3)\right)\leq S_{2}(0).\]
Similarly, $\gb_3$ will form a CVT if $S_i(1)<\frac 12(a(i)+a(i+1))<S_{i+1}(0)$ for $i=1, 2$; $\gb_4$ will form a CVT if
\begin{align*}
S_1(1)<\frac 12(a(1)+a(2))& <S_2(0)<S_2(1)<\frac 12(a(2)+a(31))<S_{31}(0)<S_{31}(1)\\
&<\frac 12(a(31)+a(32, 33))<S_{32}(0).
\end{align*}
Similarly, we can obtain the inequalities for which $\gb_5$, $\gb_6$, and $\gb_7$ will form a CVT. Due to similarity, combining all the inequalities, we see that they will be true if the following inequalities are true:
\begin{align*}
S_{1}(1)&\leq \frac{1}{2} \left(a(1)+ a(2, 3)\right)\leq S_{2}(0), \\
S_1(1)&\leq \frac{1}{2} \left(a(1)+a(21)\right)\leq S_{21}(0),\\
S_{13}(1)&\leq \frac 12(a(12, 13)+ a(21))\leq S_{21}(0),\\
S_{13}(1)&\leq \frac 12 \left(a(13)+a(21)\right)\leq S_{21}(0).
\end{align*}
Upon some simplification, we see that the above inequalities are true if $0<r\leq 2-\sqrt 3$, i.e., if $0<r\leq 0.2679491924$. If  $3^{\ell(n)} \leq n\leq 2 \cdot 3^{\ell(n)}$, then
\begin{align*}
&V(P; \gb_n(I))=\sum_{\gs \in \set{1, 2, 3}^{\ell(n)}\setminus I}\int_{J_\gs}(x-a(\gs))^2 dP+\sum_{\gs \in I} \int_{J_\gs}\min_{a\in S_\gs(\gb_2)}(x-a)^2 dP\\
&=\frac 1{3^{\ell(n)}} r^{2\ell(n)}\Big( \sum_{\gs \in \set{1, 2, 3}^{\ell(n)}\setminus I}  V+\sum_{\gs \in I}  V(P; \gb_2)\Big)\\
&=\frac 1 {3^{\ell(n)}} \cdot r^{2\ell(n)} \Big ((2\cdot 3^{\ell(n)}-n) V+(n-3^{\ell(n)}) V(P; \gb_2)\Big).
\end{align*}
Similarly, if $2\cdot 3^{\ell(n)} \leq n<3^{\ell(n)+1}$, then
\[V(P, \gb_n(I))=\frac 1 {3^{\ell(n)}} \cdot r^{2\ell(n)} \Big ((3^{\ell(n)+1}-n) V(P; \gb_2)+(n-2\cdot 3^{\ell(n)}) V(P; \gb_3)\Big).\]
 Thus, the proof of the proposition is complete.
\end{proof}

\begin{prop} \label{prop004}
Let $\gg_n(I)$ be the set given by Definition~\ref{defi23}. Then, $\gg_n(I)$ forms a CVT if $\frac{1}{79} \left(21-2 \sqrt{51}\right)\leq r\leq \frac{1}{41} \left(2 \sqrt{31}-1\right)$, i.e., if $0.08502712839\leq r\leq 0.2472080177$. Moreover, if  $3^{\ell(n)} \leq n\leq 2 \cdot 3^{\ell(n)}$, then
\[V(P, \gg_n(I))=\frac 1 {3^{\ell(n)}} \cdot r^{2\ell(n)} \Big ((2\cdot 3^{\ell(n)}-n) V+(n-3^{\ell(n)}) V(P; \gg_2)\Big),\]
and if $2\cdot 3^{\ell(n)} \leq n<3^{\ell(n)+1}$, then
\[V(P, \gg_n(I))=\frac 1 {3^{\ell(n)}} \cdot r^{2\ell(n)} \Big ((3^{\ell(n)+1}-n) V(P; \gg_2)+(n-2\cdot 3^{\ell(n)}) V(P; \gg_3)\Big).\]
\end{prop}
\begin{proof} By the definition, we have $\gg_2=\set{a(1, 21), a(22, 23, 3)}$ and $\gg_3=\set{a(1), a(2), a(3)}$. For $n\geq 3$, if $n\neq 3^{\ell(n)}$ or $n\neq 2\cdot 3^{\ell(n)}$, the subset $I$ can be chosen more than one way. This leads to the fact that if $n\neq 3^{\ell(n)}$ or $n\neq 2\cdot 3^{\ell(n)}$, the sets $\gg_n$ can be chosen multiple ways. Proceeding in the similar way, as Proposition~\ref{prop002}, let us choose
\begin{align*}
 \gg_4 &=\set{a(1), a(2), a(31, 321), a(322, 323, 33)}\\
 \gg_5&=\set{a(1), a(21, 221), a(222, 223, 23), a(31, 321), a(322, 323, 33)}\\
 \gg_6&=\set{a(11, 121), a(122, 123, 13), a(21, 221), a(222, 223, 23), a(31, 321), a(322, 323, 33)}\\
 \gg_7&=\set{a(11), a(12), a(13), a(21, 221), a(222, 223, 23), a(31, 321), a(322, 323, 33)}.
\end{align*}
Due to the same reasoning as described in the proof of Proposition~\ref{prop002}, to show $\gg_n(I)$ forms a CVT, it is enough to prove that the following inequalities are true:
\begin{align*}
S_{21}(1)&\leq \frac{1}{2} \left((a(1, 21)+a(22, 23, 3)\right)\leq S_{22}(0), \\
S_1(1)&\leq \frac{1}{2} \left(a(1) +a(21, 221)\right)\leq S_{21}(0),\\
S_{13}(1)&\leq \frac 12 (a(122, 123, 13) + a(21, 221)) \leq S_{21}(0),\\
S_{13}(1)&\leq \frac 12 \left(a(13)+a(21, 221)\right)\leq S_{21}(0).
\end{align*}
Upon some simplification, we see that the above inequalities are true if $\frac{1}{79} \left(21-2 \sqrt{51}\right)\leq r\leq \frac{1}{41} \left(2 \sqrt{31}-1\right)$, i.e.,  if $0.08502712839\leq r\leq 0.2472080177$. The rest of the proof follows in the similar way as it is given for $V(P; \gb_n)$ in Proposition~\ref{prop002}.
Thus, the proof of the proposition is complete.
\end{proof}

\begin{defi}\label{defi22} For $n\in \D N$ with $n\geq 3$ let $\ell(n)$ be the unique natural number with $3^{\ell(n)} \leq n< 3^{\ell(n)+1}$. Write
$\gd_2:=\set{a(1, 21, 221), a(222, 223, 23, 3)}$ and $\gd_3:=\set{a(1), a(2), a(3)}$. For $n\geq 3$, define $\gd_n:=\gd_n(I)$ as follows: \[\gd_n(I)=\left\{\begin{array}{cc}
\set{a(\go): \go \in \set{1, 2, 3}^{\ell(n)}\setminus I}\UU \mathop{\uu}\limits_{\go\in I} S_\go(\gd_2) & \te{ if } 3^{\ell(n)}\leq n\leq 2\cdot 3^{\ell(n)},\\
 \set{S_\go(\gd_2) : \go \in \set{1, 2, 3}^{\ell(n)}\setminus I}\UU \mathop{\uu}\limits_{\go\in I} S_\go(\gd_3) & \te{ if } 2\cdot 3^{\ell(n)}< n< 3^{\ell(n)+1},
\end{array}
\right.\] where $I \sci \set{1, 2, 3}^{\ell(n)}$ with $\te{card}(I)=n-3^{\ell(n)}$ if $3^{\ell(n)}\leq n\leq 2\cdot3^{\ell(n)}$; and
$\te{card}(I)=n-2\cdot 3^{\ell(n)}$ if $2\cdot 3^{\ell(n)}< n< 3^{\ell(n)+1}$.
\end{defi}
\begin{prop} \label{prop005}
Let $\gd_n(I)$ be the set given by Definition~\ref{defi22}. Then, $\gd_n(I)$ forms a CVT if $0.1845020699 \leq r\leq  0.2705731187$. Moreover, if  $3^{\ell(n)} \leq n\leq 2 \cdot 3^{\ell(n)}$, then
\[V(P, \gd_n(I))=\frac 1 {3^{\ell(n)}} \cdot r^{2\ell(n)} \Big ((2\cdot 3^{\ell(n)}-n) V+(n-3^{\ell(n)}) V(P; \gd_2)\Big),\]
and if $2\cdot 3^{\ell(n)} \leq n<3^{\ell(n)+1}$, then
\[V(P, \gd_n(I))=\frac 1 {3^{\ell(n)}} \cdot r^{2\ell(n)} \Big ((3^{\ell(n)+1}-n) V(P; \gd_2)+(n-2\cdot 3^{\ell(n)}) V(P; \gd_3)\Big).\]
\end{prop}
\begin{proof}
By the definition, we have $\gd_2=\set{a(1, 21, 221), a(222, 223, 23, 3)}$ and $\gd_3=\set{a(1), a(2), a(3)}$. For $n\geq 3$, if $n\neq 3^{\ell(n)}$ or $n\neq 2\cdot 3^{\ell(n)}$, the subset $I$ can be chosen more than one way. This leads to the fact that if $n\neq 3^{\ell(n)}$ or $n\neq 2\cdot 3^{\ell(n)}$, the sets $\gd_n$ can be chosen multiple ways. Proceeding in the similar way, as Proposition~\ref{prop002}, let us choose
\begin{align*}
 \gd_4 &=\set{a(1), a(2), a(31, 321, 3221),  a(3222, 3223, 323, 33)}\\
 \gd_5&=\set{a(1), a(21, 221, 2221), a(2222, 2223, 223, 23), \\
 &\qquad \qquad \qquad a(31, 321, 3221), a(3222, 3223, 323, 33)}\\
 \gd_6&=\set{a(11, 121, 1221), a(1222, 1223, 123, 13), a(21, 221, 2221), a(2222, 2223, 223, 23), \\
 &\qquad \qquad \qquad a(31, 321, 3221), a(3222, 3223, 323, 33)}\\
 \gd_7&=\set{a(11), a(12), a(13), a(21, 221, 2221), a(2222, 2223, 223, 23), \\
 &\qquad \qquad \qquad a(31, 321, 3221), a(3222, 3223, 323, 33)}.
\end{align*}
Due to the same reasoning as described in the proof of Proposition~\ref{prop002}, to show $\gd_n(I)$ forms a CVT, it is enough to prove that the following inequalities are true:
\begin{align*}
S_{221}(1)&\leq \frac{1}{2} \left(a(1, 21, 221)+ a(222, 223, 23, 3)\right)\leq S_{222}(0), \\
S_1(1)&\leq \frac{1}{2} \left(a(1)+a(21, 221, 2221)\right)\leq S_{21}(0),\\
S_{13}(1)&\leq \frac 12(a(1222, 1223, 123, 13) +a(21, 221, 2221))\leq S_{21}(0),\\
S_{13}(1)&\leq \frac 12 \left(a(13)+a(21, 221, 2221)\right)\leq S_{21}(0).
\end{align*}
 The above inequalities are true if $0.1845020699 \leq  r \leq 0.2705731187$. The rest of the proof follows in the similar way as it is given for $V(P; \gb_n(I))$ in Proposition~\ref{prop002}. Thus, the proof of the proposition is complete.
\end{proof}

The following proposition is useful to establish Lemma~\ref{lemmaR001}, and Lemma~\ref{lemmaR1}.
\begin{prop} \label{prop0001}
Let $\gk:=\set{a_1, a_2}$, where $a_1:=E(X : X\in [0, \frac 12])$, and $a_2:=E(X : X\in [\frac 12, 1])$. Then, $a_1=\frac{r+1}{6-2r}$, and $a_2=\frac{5-3 r}{6-2 r}$, and the corresponding distortion error is given by
\[V(P; \gk)=\frac{-7 r^3+13 r^2-9 r+3}{6 (r-3)^2 (r+1)}.\] \end{prop}
\begin{proof}
By the hypothesis, we have
\begin{align*}
a_1&=E(X : X \in [0, \frac 12])=E\Big(X : X\in  J_1\uu J_{21}\uu J_{221}\uu \cdots \Big), \te{ and } \\
a_2&=E(X : X \in [\frac 12,  1])=E\Big(X : X\in  J_3\uu J_{23}\uu J_{223}\uu \cdots \Big),
\end{align*}
yielding
\[a_1=2 \sum_{n=1}^\infty \frac 1 {3^n}\frac 12 (-r^{n-1}+r^n+1) =\frac{r+1}{6-2r}, \te{ and }  a_2=2 \sum_{n=1}^\infty \frac 1 {3^n}\frac 12 (r^{n-1}-r^n+1)=\frac{5-3 r}{6-2 r},\]
 and the corresponding distortion error is given by
\begin{align*}
& V(P; \gk) = 2 \int_{J_1\uu J_{21}\uu J_{221}\uu J_{2221}\cdots} \Big(x-\frac{r+1}{6-2r}\Big)^2 dP
\end{align*}
implying \[V(P; \gk)=2 \Big(\sum _{n=1}^{\infty } \frac{r^{2 n}}{3^n} V+\sum _{n=1}^{\infty } \frac 1 {3^n} \Big(\frac{1}{2} \left(-r^{n-1}+r^n+1\right)-\frac{r+1}{6-2r}\Big)^2 \Big)=\frac{-7 r^3+13 r^2-9 r+3}{6 (r-3)^2 (r+1)}.\]
Thus, the proposition is yielded.
\end{proof}

\section{Optimal sets of $n$-means and the $n$th quantization errors for $r=\frac 1{25}$} \label{sec2}
Let $\gb_n$ be the set given by Definition~\ref{defi00}. In this section, we show that for all $n\geq 2$, the sets $\gb_n$ form the optimal sets of $n$-means for $r=\frac 1{25}$. To calculate the distortion errors we will frequently use the formula given by Corollary~\ref{cor1}.
Notice that by Lemma~\ref{lemma2}, in this case, we have $E(X)=\frac 12 \te{ and }  V:=V(X)=\frac{1-r}{6 (r+1)}=\frac{2}{13}$.
\begin{lemma}  \label{lemmaR001}
The set $\gb:=\set{a(1), a(2, 3)}$ forms the optimal set of two-means, and the corresponding quantization error is given by $V_2=\frac{314}{8125}=0.0386462$.
\end{lemma}

\begin{proof} Let $\gb:=\set{a_1, a_2}$ be an optimal set of two-means. Since the points in an optimal set are the conditional expectations in their own Voronoi regions, without any loss of generality, we can assume that $0<a_1<a_2<1$.
Let us consider the set $\gk:=\set{a(1), a(2, 3)}$. The distortion error due to the set $\gk$ is given by
\begin{equation} \label{eqR1} V(P; \gk)=\int_{J_1}(x-a(1))^2 dP+\int_{J_2\uu J_3}(x-a(2, 3))^2 dP=0.0386462.
\end{equation}
Since $V_2$ is the quantization for two-means, we have $V_2\leq 0.0386462$. Assume that $0.38<a_1$. Then,
\[V_2\geq \int_{J_1}(x-0.38)^2 dP=0.0432821>V_2,\]
which is a contradiction. Hence, $a_1\leq 0.38$. Similarly, $0.62\leq a_2$. Since $\frac 12 (a_1+a_2)\leq \frac 12(0.38+1)=0.69<S_3(0)=0.96$, the Voronoi region of $a_1$ does not contain any point from $J_3$. Similarly, the Voronoi region of $a_2$ does not contain any point from $J_1$. Since the union of the Voronoi regions of $a_1$ and $a_2$ covers $J_1\uu J_2\uu J_3$, without any loss of generality, we can assume that the Voronoi region of $a_2$ contains points from $J_2$, and $\frac 12(a_1+a_2)\leq \frac 12$. If $\frac 12(a_1+a_2)=\frac 12$, then substituting $r=\frac 1{25}$, by Proposition~\ref{prop0001}, we have
\[V_2=\frac{866}{17797}=0.0486599>V_2,\]
which leads to a contradiction. Hence, we can conclude that $\frac 12(a_1+a_2)<\frac 12$. Using the similar technique as it is given in the proof of Lemma~3.1 in \cite{R2}, we can show that $S_{1}(1)\leq \frac 12(a_1+a_2)\leq S_2(0)$ yielding the fact that $a_1=a(1)$, $a_2=a(2, 3)$, and $V_2=\frac{314}{8125}=0.0386462$. Hence, the proof of the lemma is complete.
\end{proof}

 \begin{lemma}  \label{lemmaR002}
The set $\gb:=\set{a(1), a(2), a(3)}$ forms an optimal set of three-means, and the corresponding quantization error is given by $V_3=\frac{2}{8125}=0.000246154$.
\end{lemma}

\begin{proof}
Consider the set of three points $\gk:=\set{a(1), a(2), a(3)}$. The distortion error due to the set $\gk$ is given by
\[V(P; \gk)=\sum_{j=1}^3 \int_{J_j} (x-a(j))^2 dP=\frac{2}{8125}=0.000246154.\]
Since $V_3$ is the quantization error for three-means, we have $V_3\leq 0.000246154$. Let $\gb:=\set{a_1, a_2, a_3}$, where $0<a_1<a_2<a_3<1$, be an optimal set of three-means. If $S_1(1)=\frac 1{25}<\frac 1{23}<a_1$, then
\[V_3\geq \int_{J_1}(x-\frac 1{23})^2 dP=\frac{13709}{51577500}=0.000265794>V_3,\]
which gives a contradiction. Thus, we can assume that $a_1\leq \frac 1{23}$. Similarly, $\frac {22}{23}\leq a_3$. Suppose that $\gb\ii J_1=\es$. Then, due to symmetry, we can assume that $\gb\ii J_3=\es$, and then
\[V_3\geq 2\int_{J_1}(x-a_1)^2 dP=2\int_{J_1}(x-S_1(1))^2 dP=\frac{7}{16250}=0.000430769>V_3,\]
which leads to a contradiction. So, we can assume that $\gb\ii J_1\neq \es$, i.e., $a_1<S_1(1)$. Similarly, $\gb\ii J_3\neq \es$, i.e., $S_3(0)<a_3$. Now, we show that $\gb \ii J_2\neq \es$. Suppose that $\gb\ii J_2=\es$. Then, either $a_2<\frac {12}{25}=S_2(0)$, or $ \frac {13}{25}=S_2(1)<a_2$. First, assume that $a_2<S_2(0)$. Then, notice that $S_2(1)=\frac {13}{25}<\frac 12(S_2(0)+S_3(0))<S_3(0)$ yielding the fact that the Voronoi region of $S_2(0)$ contains $J_2$. Hence,
\[V_3\geq \int_{J_2}(x-S_2(0))^2 dP+\int_{J_3}(x-a(3))^2 dP=\frac{29}{97500}=0.000297436>V_3,\]
which is a contradiction. Similarly, we can show that a contradiction arises if $\frac {13}{25}=S_2(1)<a_2$. Thus, we can assume that $\gb\ii J_2\neq \es$. Now, if the Voronoi region of $a_1$ contains points from $J_2$, we have $\frac 12(a_1+a_2)>\frac {12}{25}=S_2(0)$ implying $a_2>\frac {24}{25}-a_1\geq \frac {24}{25}-\frac 1{25}=\frac {23}{25}>S_2(1)$, which is a contradiction as $\gb\ii J_2\neq \es$. Hence, we can assume that the Voronoi region of $a_1$ does not contain any point from $J_2$, and so from $J_3$. Similarly, we can show that the Voronoi region of $a_2$ does not contain any point from $J_1$ and $J_3$, and the Voronoi region of $a_3$ does not contain any point from $J_2$, and so from $J_1$. Thus, by Proposition~\ref{prop0}, we conclude that $a_1=a(1)$, $a_2=a(2)$, and $a_3=a(3)$, and the corresponding quantization error is given by $V_3=\frac{2}{8125}=0.000246154$, which is the lemma.
\end{proof}

\begin{prop}\label{propR003}
Let $\gb_n$ be an optimal set of $n$-means for any $n\geq 3$. Then, $\gb_n\ii J_j\neq \es$ for all $1\leq j\leq 3$, and $\gb_n$ does not contain any point from the open intervals $(S_1(1), S_2(0))$ and $(S_2(1), S_3(0))$. Moreover, the Voronoi region of any point in $\gb_n\ii J_j$ does not contain any point from $J_i$, where $1\leq i\neq j\leq 3$.
\end{prop}

\begin{proof}
By Lemma~\ref{lemmaR002}, the proposition is true for $n=3$. Let us prove the lemma for $n\geq 4$.
Let $\gb_n:=\set{a_1, a_2, \cdots, a_n}$ be an optimal set of $n$-means for $n\geq 4$. Since the points in an optimal set are the conditional expectations in their own Voronoi regions, without any loss of generality, we can assume that $0<a_1<a_2<\cdots<a_n<1$.  Consider the set of four elements $\gk:=S_1(\gb_2)\uu\set{a(2), a(3)}$. Then,
\[V(P; \gk)=\int_{J_1} \min_{a\in S_1(\gb_2)} (x-a)^2 dP +\int_{J_2}(x-a(2))^2 dP+\int_{J_3}(x-a(3))^2 dP=\frac{938}{5078125}=0.000184714.\]
Since $V_n$ is the quantization error for $n$-means for $n\geq 4$, we have $V_n\leq V_4\leq 0.000184714$. Suppose that $S_1(1)\leq  a_1$. Then,
\[V_n\geq \int_{J_1}(x-S_1(1))^2 dP=\frac{7}{32500}=0.000215385>V_n,\]
which is a contradiction. So, we can assume that $a_1<S_1(1)$, i.e., $\gb_n\ii J_1\neq \es$. Similarly, $\gb_n\ii J_3\neq \es$. We now show that $\gb_n\ii J_2\neq \es$. For the sake of contradiction, assume that $\gb_n\ii J_2=\es$. Let $a_j:=\max\set{a_i : a_i<S_2(0) \te{ for } 1\leq i\leq n-1}$. Then, $a_j<S_2(0)$. As $\gb_n\ii J_2=\es$, we have $ S_2(1)<a_{j+1}$. If $a_j<\frac 12(S_1(1)+S_2(0))=\frac {13}{50}$, then as $\frac 12(a_j+a_{j+1})<\frac 12(\frac {13}{50}+S_2(1))=\frac{39}{100}<\frac{12}{25}=S_2(0)$, we have
\[V_n\geq \int_{J_2}(x-S_2(1))^2 dP=\frac{7}{32500}=0.000215385>V_n,\]
which leads to a contradiction. So, we can assume that $\frac {13}{50}\leq a_j<S_2(0)$. Then, by Proposition~\ref{prop0}, we have $\frac 12(a_{j-1}+a_j)<\frac 1{25}$ implying $a_{j-1}<\frac 2{25}-a_j\leq \frac2{25}-\frac {13}{50}=-\frac 9{50}<0$, which gives a contradiction as $\gb_n\ii J_1\neq \es$. Hence, we can conclude that $\gb_n\ii J_2\neq \es$. Notice that $(S_1(1), S_2(0))=(\frac 1{25}, \frac {12}{25})$.
Suppose that $\gb_n$ contains a point from the open interval $(\frac 1{25}, \frac {12}{25})$. Let $a_j:=\max\set{a_i : a_i<\frac 1{25} \te{ for } 1\leq i\leq n-2}$. Then, due to Proposition~\ref{prop0}, $a_{j+1}\in (\frac 1{25}, \frac {12}{25})$, and $a_{j+2}\in J_2$. The following cases can arise:

Case~1. $\frac 1{25} <a_{j+1}\leq \frac {13}{50}$.

Then, $\frac12(a_{j+1}+a_{j+2})>\frac {12}{25}$ implying $a_{j+2}>\frac {24}{25}-a_{j+1}\geq \frac {24}{25}-\frac {13}{50}=\frac {35}{50}>S_{2}(1)$,
which leads to a contradiction because $a_{j+2}\in J_2$.

Case~2. $  \frac {13}{50}\leq a_{j+1}<\frac {12}{25}$.

Then, $\frac 12(a_j+a_{j+1})<\frac 1{25}$ implying $a_j\leq \frac 2{25}-a_{j+1}\leq \frac 2{25}-\frac {13}{50}=-\frac 9{50}$, which is a contradiction because $a_j>0$.

Thus, by Case~1 and Case~2, we can conclude that $\gb_n$ does not contain any point from the open interval $(S_1(1), S_2(0))$. Reflecting the situation with respect to the point $\frac 12$, we can conclude that $\gb_n$ does not contain any point from the open interval $(S_2(1), S_3(0))$ as well. To prove the last part of the proposition, we proceed as follows: Let $a_j:=\max\set{a_i : a_i<\frac 1{25} \te{ for } 1\leq i\leq n-2}$. Then, $a_j$ is the rightmost element in $\gb_n\ii J_1$, and $a_{j+1}\in \gb_n\ii J_2$. Suppose that the Voronoi region of $a_j$ contains points from $J_2$. Then, $\frac 12(a_j+a_{j+1})>\frac {12}{25}$ implying $a_{j+1}>\frac {24}{25}-a_j\geq \frac {24}{25}-\frac 1{25}=\frac {23}{25}>S_2(1)$, which yields a contradiction as $a_{j+1}\in J_2$. Thus, the Voronoi region of any point in $\gb_n\ii J_1$ does not contain any point from $J_2$, and $J_3$ as well. Similarly, we can prove that the Voronoi region of any point in $\gb_n\ii J_2$ does not contain any point from $J_1$ and $J_3$, and the Voronoi region of any point in $\gb_n\ii J_3$ does not contain any point from $J_1$ and $J_2$. Thus, the proof of the proposition is complete.
\end{proof}

 The following lemma is a modified version of Lemma~4.5 in \cite{GL2}, and the proof follows similarly. One can also see Lemma~3.5 in \cite{R2}.

\begin{lemma} \label{lemmaR004}
Let $n\geq 3$, and let $\gb_n$ be an optimal set of $n$-means such that $\gb_n\ii J_j\neq \es$ for all $1\leq j\leq 3$, and $\gb_n$ does not contain any point from the open intervals $(S_1(1), S_2(0))$ and $(S_2(1), S_3(0))$. Further assume that the Voronoi region of any point in $\gb_n\ii J_j$ does not contain any point from $J_i$, where $1\leq i\neq j\leq 3$. Set $\gk_j:=\gb_n\ii J_j$, and $n_j:=\te{card }(\gk_j)$ for $1\leq j\leq 3$.
 Then, $S_j^{-1}(\gk_j)$ is an optimal set of $n_j$-means, and
$V_{n}=\frac 1{1875}(V_{n_1}+V_{n_2}+V_{n_3}).$
\end{lemma}

Let us now state and prove the following theorem which gives the optimal sets of $n$-means for all $n\geq 3$, where $r=\frac 1{25}$.
 \begin{theorem}  \label{Th1}
Let $P$ be the probability measure on $\D R$ with support the Cantor set $C$ generated by the three contractive similarity mappings $S_j$ for $j=1,2, 3$. Let $n\in \D N$ with $n\geq 3$. Take $r=\frac 1{25}$. Then, the sets $\gb_n:=\gb_n(I)$ given by Definition~\ref{defi00} form the optimal sets of $n$-means for $P$ with the corresponding quantization error $V_n:=V(P;\gb_n(I))$, where $V(P; \gb_n(I))$ is given by Proposition~\ref{prop002}.
\end{theorem}

\begin{proof}
We will proceed by induction on $\ell(n)$. If $n=3$, then by Lemma~\ref{lemmaR002}, the theorem is true. Now, we show that the theorem is true if $n=4$. Let   $\gk_j:=\gb_n\ii J_j$, and $n_j:=\te{card }(\gk_j)$ for $1\leq j\leq 3$. Since $S_j^{-1}(\gk_j)$ is an optimal set of $n_j$-means for $1\leq j\leq 3$, and for $n=4$ the possible choices for the triplet $(n_1, n_2, n_3)$ are $(2,1,1)$, $(1, 2, 1)$, and $(1, 1, 2)$, by Proposition~\ref{propR003} and Lemma~\ref{lemmaR004}, the set $\gb_4$ forms an optimal set of four-means with quantization error $V(P; \gb_4)$ given by Proposition~\ref{prop002}. Remember that for a given $n$, among all the possible choices of the triplets $(n_1, n_2, n_3)$, the triplets $(n_1, n_2, n_3)$ which give the smallest distortion error will give the optimal sets of $n$-means. Notice that for $n=5$, the possible choices of the triplets are $(3, 1, 1)$, $(1, 3, 1)$, $(1,1,3)$, $(1, 2, 2)$, $(2, 1, 2)$, $(2, 2, 1)$ among which $(1, 2, 2), \, (2, 1, 2), \, (2, 2, 1)$ give the smallest distortion error. Hence, the optimal sets of five-means are $\set{a(1)} \uu S_2(\gb_2)\uu S_3(\gb_2)$, $S_1(\gb_2)\uu \set{a(2)}\uu S_3(\gb_2)$, and $S_1(\gb_2)\uu S_2(\gb_2)\uu \set{a(3)}$ which are the sets $\gb_5$ given by Definition~\ref{defi00}. Similarly, we can calculate the optimal sets of six- and seven-means. Thus, the theorem is true for $\ell(n)=1$. Let us assume that the theorem is true for all $\ell(n)<m$, where $m\in \D N$ and $m\geq 2$. We now show that the theorem is true if $\ell(n)=m$. Let us first assume that $3^m\leq n\leq 2\cdot 3^m$.
Let $\gb_n$ be an optimal set of $n$-means for $P$ such that $3^m\leq n\leq 2\cdot 3^m$. Let $\te{card }(\gb_n\ii J_j)=n_j$ for $j=1, 2,3$, and then by Lemma~\ref{lemmaR004}, we have
\begin{equation} \label{eq34} V_{n}=\frac 1{1875}(V_{n_1}+V_{n_2}+V_{n_3}).
\end{equation}
 Without any loss of generality, we can assume that $n_1\geq n_2\geq n_3$. Let $u, v, w\in \D N$ be such that
\begin{equation}\label{eq35} 3^u\leq n_1\leq 2\cdot 3^u, \   3^v\leq n_2\leq 2\cdot 3^v, \te{ and } 3^w\leq n_3\leq 2\cdot 3^w.
\end{equation}
Proceeding in the similar lines as the proof of Theorem~3.6 in \cite{R2}, we can show that $u=v=w=m-1$. Since by Lemma~\ref{lemmaR004}, for $S_j^{-1}(\gb_n\ii J_j)$ is an optimal set of $n_j$ means where $3^{m-1}\leq n_j\leq 2\cdot 3^{m-1}$, we have
\[S_j^{-1}(\gb_n\ii J_j)=\set{a(\go) : \go \in \set{1, 2, 3}^{m-1}\setminus I_j} \uu \left(\uu_{\go \in I_j} S_\go(\gb_2)\right),\]
where $I_j\ci \set{1, 2, 3}^{m-1}$ with $\te{card }(I_j)=n_j-3^{m-1}$ for $1\leq j\leq 3$. Hence,
\[\gb_n:=\gb_n(I)=\UU_{j=1}^3 S_j^{-1}(\gb_n\ii J_j)=\set{a(\go) : \go \in \set{1, 2, 3}^{\ell(n)}\setminus I} \uu \left(\uu_{\go \in I} S_\go(\gb_2)\right),\]
where $I\ci \set{1, 2, 3}^{m}$ with $\te{card }(I)=n-3^{m}$, is an optimal set of $n$-means. The corresponding quantization error is
\[V_n=\frac 1{3^{m}} r^{2m}\left((2\cdot 3^{m}-n)V+ (n-3^{m})V_2\right)=V(P; \gb_n(I)),\]
where $V(P; \gb_n(I))$ is given by Proposition~\ref{prop002}. Thus, the theorem is true if $3^m\leq n\leq 2\cdot 3^m$. Similarly, we can prove that the theorem is true if $2\cdot 3^{m}< n< 3^{m+1}$. Hence, by the induction principle, the proof of the theorem is complete.
\end{proof}

\section{Optimal sets of $n$-means and the $n$th quantization errors for $r=r_0$ and $r=r_1$} \label{sec3}

In this section, we give the proof of Theorem~\ref{Th2}.
First, we prove the following two lemmas.
\begin{lemma}  \label{lemmaR1}
Let $r_0$ and $r_1$ be the real numbers given by Theorem~\ref{Th2}. Then,
the set $\gg:=\set{a(1, 21), a(22, 23, 3)}$ for $r=r_0$ and $r=r_1$ form the optimal sets of two-means, and the corresponding quantization errors are, respectively, given by $V_2=0.0324042$, and $V_2=0.026897$.
\end{lemma}

\begin{proof}
First, we prove that $\gg$ forms an optimal set of two-means for $r=r_0$. Let $\gg:=\set{a_1, a_2}$ be an optimal set of two-means. Since, the points in an optimal set are the expected values of their own Voronoi regions, without any loss of generality, we can assume that $0<a_1<a_2<1$.
Let us consider the set $\gk:=\set{a(1, 21), a(22, 23, 3)}$. The distortion error due to the set $\gk$ is given by
\begin{equation} \label{eqR1} V(P; \gk)=\int_{J_1}(x-a(1, 21))^2 dP+\int_{J_2\uu J_3}(x-a(22, 23, 3))^2 dP=0.0324042.
\end{equation}
Since $V_2$ is the quantization error for two-means, we have $V_2\leq 0.0324042$. Assume that $0.39<a_1$. Then,
\[V_2\geq \int_{J_1}(x-0.39)^2 dP=0.0328529>V_2,\]
which is a contradiction. Hence, $a_1\leq 0.39$. Similarly, $0.61\leq a_2$. Since $\frac 12 (a_1+a_2)\leq \frac 12(0.39+1)=0.695<S_3(0)=0.837722$, the Voronoi region of $a_1$ does not contain any point from $J_3$. Similarly, the Voronoi region of $a_2$ does not contain any point from $J_1$. Since the union of the Voronoi regions of $a_1$ and $a_2$ covers $J_1\uu J_2\uu J_3$, without any loss of generality, we can assume that the Voronoi region of $a_2$ contains points from $J_2$, and $\frac 12(a_1+a_2)\leq \frac 12$. If $\frac 12(a_1+a_2)=\frac 12$, then substituting $r=0.1622776602$, by Proposition~\ref{prop0001}, we have
\[V(P; \gk)=0.0329779,\]
which contradicts \eqref{eqR1}. Hence, we can conclude that $\frac 12(a_1+a_2)<\frac 12$. Using the similar technique as it is given in the proof of Lemma~3.1 in \cite{R2}, we can show that either $\frac 1 2(a_1+a_2)=\frac 12(a(1, 21)+a(22, 23, 3))=0.466886$, or $\frac 12(a_1+a_2)=\frac 12(a(1)+a(2, 3))=0.395285$, i.e., either $S_{21}(1)<\frac 12(a_1+a_2)<S_{22}(0)$, or $S_1(1)<\frac 12(a_1+a_2)<S_2(0)$. Notice that if  $S_{21}(1)<\frac 12(a_1+a_2)<S_{22}(0)$, then $\gg_2$, given by Definition~\ref{defi23}, forms the optimal set of two-means. On the other hand, if $S_1(1)<\frac 12(a_1+a_2)<S_2(0)$, then $\gb_2$, given by Definition~\ref{defi00}, forms the optimal set of two-means. In fact, later we will see that $V(P; \gg_2)=V(P; \gb_2)=0.0324042$ for $r=0.1622776602$. Thus, $\gg_2$ forms the optimal set of two-means for $r=r_0$ with quantization error $V_2=0.0324042$. Similarly, we can show that $\gg_2$ forms the optimal set of two-means if $r=r_1$ with quantization error $V_2=0.026897$. Hence, the lemma is yielded.
\end{proof}

The following lemma is true analogously as Lemma~3.3 in \cite{R2}.

\begin{lemma}  \label{lemmaR2}
The set $\gg_3:=\set{a(1), a(2), a(3)}$ for $r=r_0$, and $r=r_1$ form the optimal sets of three-means, and the corresponding quantization errors are, respectively, given by $V_3=0.00316342$, and $V_3=0.00558347$.
\end{lemma}

The following proposition is true analogously as Proposition~3.5 in \cite{R2}.

\begin{prop}\label{propR3}
Let $n\geq 3$, and let $\gg_n$ be an optimal set of $n$-means for $r=r_0$, and $r=r_1$. Then, $\gg_n\ii J_j\neq \es$ for all $1\leq j\leq 3$, and $\gg_n$ does not contain any point from the open intervals $(S_1(1), S_2(0))$ and $(S_2(1), S_3(0))$. Moreover, the Voronoi region of any point in $\gg_n\ii J_j$ does not contain any point from $J_i$, where $1\leq i\neq j\leq 3$.
\end{prop}

The following remark is true due to Proposition~\ref{propR3}.
\begin{remark} \label{remarkR4}
Let $n\geq 3$, and let $\gg_n$ be an optimal set of $n$-means for $r=r_0$, and $r=r_1$. Set $\gk_j:=\gg_n\ii J_j$, and $n_j:=\te{card }(\gk_j)$ for $1\leq j\leq 3$. Then, $S_j^{-1}(\gk_j)$ is an optimal set of $n_j$-means, and for $r=r_0$ and $r=r_1$, respectively, we have
$V_{n}=\frac 1{3}r_0^n(V_{n_1}+V_{n_2}+V_{n_3})$ and $V_{n}=\frac 1{3}r_1^n(V_{n_1}+V_{n_2}+V_{n_3})$.
\end{remark}

\subsection*{Proof of Theorem~\ref{Th2}}
We proceed to prove it by induction on $\ell(n)$. By Lemma~\ref{lemmaR2}, we see that the theorem is true for $n=3$. Proceeding in the similar way, as mentioned in the proof of Theorem~\ref{Th1}, we can show that for $n=4, 5, 6, 7$, the sets $\gg_n$ form the optimal sets of $n$-means for $r=r_0$ and $r=r_1$. Thus, the theorem is true if $\ell(n)=1$.  Let us assume that the theorem is true for all $\ell(n)<m$, where $m\in \D N$ and $m\geq 2$. We now show that the theorem is true if $\ell(n)=m$. Let us first assume that $3^m\leq n\leq 2\cdot 3^m$.
Let $\gg_n$ be an optimal set of $n$-means for $P$ such that $3^m\leq n\leq 2\cdot 3^m$. Let $\te{card }(\gg_n\ii J_j)=n_j$ for $j=1, 2,3$, and then by Remark~\ref{remarkR4}, we have
\begin{equation*} \label{eq341} V_{n}=\frac 1{3}r_0^n(V_{n_1}+V_{n_2}+V_{n_3}) \te{ for } r=r_0, \te{ and } V_{n}=\frac 1{3}r_1^n(V_{n_1}+V_{n_2}+V_{n_3}) \te{ for } r=r_1.
\end{equation*}
The rest of the proof for $r=r_0$ and $r=r_1$ follow in the similar way as the proof of Theorem~\ref{Th1}. Thus, we complete the proof of the theorem.
\qed

\section{Main results}

The two theorems in this section, state and prove the main results of the paper.

\begin{theorem} \label{Th3} Let $r_0, r_1\in (0, \frac 13)$ be the unique real numbers satisfying, respectively, the equations
\begin{align*} -\frac{3 r^5+15 r^4+6 r^3-42 r^2+31 r-13}{240 (r+1)}&=-\frac{3 r^3-3 r^2+r-1}{24 (r+1)},\\
-\frac{3 r^5+15 r^4+6 r^3-42 r^2+31 r-13}{240 (r+1)}&=-\frac{3 r^7+15 r^6+60 r^5+66 r^4+18 r^3-324 r^2+283 r-121}{2184 (r+1)}.
\end{align*}
Then, $r_0=0.1622776602$, and $r_1=0.2317626315$. Let the sets $\gb_n$ and $\gg_n$ be, respectively, given by Definition~\ref{defi00}, and Definition~\ref{defi23}. Then, $\gb_n$ form the optimal sets of $n$-means for $0<r\leq r_0$, and $\gg_n$  forms the optimal sets of $n$-means for $r_0\leq r\leq r_1$.
\end{theorem}

\begin{proof}
By Proposition~\ref{prop002}, Proposition~\ref{prop004}, and Proposition~\ref{prop005}, we see that both $\gb_n$ and $\gg_n$ form CVTs if $0.08502712839\leq r\leq 0.2472080177$; both $\gg_n$ and $\gd_n$ form CVTs if $0.1845020699 \leq  r\leq  0.2472080177$; both $\gb_n$ and $\gd_n$ form CVTs if $0.1845020699 \leq r\leq 0.2679491924$.
Again, $V(P; \gb_3)=V(P; \gg_3)=V(P; \gd_3)$. Thus, for any $3^{\ell(n)}\leq n<3^{\ell(n)+1}$, from the aforementioned propositions, in the case of $V(P; \gb_n(I))$ and $V(P; \gg_n(I))$, we see that $V(P; \gb_n(I))> V(P; \gg_n(I))$, $V(P; \gb_n(I))= V(P; \gg_n(I))$, and $V(P; \gb_n)< V(P; \gg_n)$ will be true if $V(P; \gb_2)> V(P; \gg_2)$, $V(P; \gb_2)= V(P; \gg_2)$, and $V(P; \gb_2)< V(P; \gg_2)$, respectively. Similarly, it hold in the case of $V(P; \gb_n)$ and $V(P; \gd_n)$, and in the case of $V(P; \gg_n)$ and $V(P; \gd_n)$. Next, we have
\begin{align*}
V(P; \gb_2)&=-\frac{3 r^3-3 r^2+r-1}{24 (r+1)}, \\
V(P; \gg_2)&=-\frac{3 r^5+15 r^4+6 r^3-42 r^2+31 r-13}{240 (r+1)},\\
V(P; \gd_2)&=-\frac{3 r^7+15 r^6+60 r^5+66 r^4+18 r^3-324 r^2+283 r-121}{2184 (r+1)}.
 \end{align*}
After some calculation, we observe that
$V(P; \gb_2)<V(P; \gg_2)$ is true  if $0.08502712839\leq r<0.1622776602$; $V(P; \gb_2)=V(P; \gg_2)$ if $r=0.1622776602$, and $V(P; \gb_2)>  V(P; \gg_2)$ if $0.1622776602< r\leq 0.2472080177$. Again, $V(P; \gb_2)>V(P; \gd_2)$ if $0.1701473031 < r\leq 0.2679491924$ and $V(P; \gb_2)=V(P; \gd_2)$ if $r=0.1701473031$. Recall that the sets $\gb_n$ form CVTs if $0<r\leq 0.2679491924$. Hence, we can say that the sets $\gb_n$ do not form the optimal sets of $n$-means if $0.1622776602<r\leq 0.2679491924$. In Theorem~\ref{Th2},  we have seen that the sets $\gb_n$ form the optimal sets of $n$-means if $r=\frac 1{25}$. Using the similar technique, we can show that the sets $\gb_n$ form the optimal sets of $n$-means if $0<r\leq \frac 1{25}$. Since $V(P; \gb_2)=V(P; \gg_2)$ if $r=r_0$; and by Theorem~\ref{Th2}, the sets $\gg_n$ form the optimal sets of $n$-means if $r=r_0$, we can say that the sets $\gb_n$ also form the optimal sets of $n$-means if $r=r_0$. Again, $V(P; \gb_2)$ is strictly decreasing in the closed interval $[0, r_0]$. Hence, the sets $\gb_n$ form the optimal sets of $n$-means for $0<r\leq r_0$.

To prove the remaining part of the theorem, we see that

$(i)$ $V(P; \gb_2)<V(P; \gg_2)$  if $0.08502712839\leq r<0.1622776602$; $V(P; \gb_2)=V(P; \gg_2)$ if $r=0.1622776602$, and $V(P; \gb_2)>  V(P; \gg_2)$ if $0.1622776602< r\leq 0.2472080177$.

$(ii)$ $V(P; \gd_2)<V(P; \gg_2)$ if  $0.2317626315<r\leq 0.2472080177$; $V(P; \gd_2)=V(P; \gg_2)$ if $r=0.2317626315$, and $V(P; \gd_2)>  V(P; \gg_2)$ if $0.1845020699\leq r< 0.2317626315$.

Thus, the sets $\gg_n$ do not form the optimal sets of $n$-means if $0.08502712839\leq r<0.1622776602$, or if $0.2317626315<r\leq 0.2472080177$; in other words, the range of $r$ values for which the sets $\gg_n$ form the optimal sets of $n$-means is bounded below by $r_0=0.1622776602$ and bounded above by $r_1=0.2317626315$.
By Theorem~\ref{Th2}, we see that the sets $\gg_n$ form the optimal sets of $n$-means if $r=r_0$, and $r=r_1$. Again, $V(P; \gg_2)$ is strictly decreasing in the closed interval $[r_0, r_1]$. Hence, the precise range of $r$ values  for which the sets $\gg_n$ form the optimal sets of $n$-means is given by $r_0\leq r\leq r_1$. Thus, the proof of the theorem is complete.
\end{proof}
Since the Cantor set $C$ under investigation satisfies the strong separation condition, with each $S_j$ having contracting factor of $r$, the Hausdorff dimension of the Cantor set is equal to the similarity dimension.  Hence, from the equation $3(r)^\gb=1$,  we have $\dim_{\te{H}}(C)=\beta =-\frac {\log 3}{\log r}$.  By Theorem~14.17 in \cite{GL1}, the quantization dimension $D (P) $ exists and is equal to $\gb$. In Theorem~\ref{Th4}, we show that $\gb$ dimensional quantization coefficient for $P$ does not exist.

\begin{theorem} \label{Th4}
The $\gb$-dimensional quantization coefficient for $0<r\leq r_1$ does not exist.
\end{theorem}
\begin{proof} We have $3^{\frac 1 \gb}=\frac 1 r$. Notice that $\Big\{\Big(3^{\ell(n)}\Big)^{\frac 2 {\gb}}V_{3^{\ell(n)}}(P)\Big\}$ and $\Big\{\Big(2\cdot 3^{\ell(n)}\Big)^{\frac 2 {\gb}}V_{2\cdot 3^{\ell(n)}}(P)\Big\}$ are two different subsequences of the sequence $\Big\{n^{\frac 2 {\gb}}V_{n}(P)\Big\}$. First, assume that $0<r\leq r_0$. Then, by Theorem~\ref{Th3}, $\gb_n$ is an optimal set of $n$-means for $0<r\leq r_0$. Recall Proposition~\ref{prop002}. Then, we have
\begin{equation} \label{eq00}
\lim_{n\to \infty} \Big(3^{\ell(n)}\Big)^{\frac 2 {\gb}}V_{3^{\ell(n)}}(P)=\lim_{n\to \infty} \frac 1{r^{2\ell(n)}} \frac 1 {3^{\ell(n)}} r^{2\ell(n)} 3^{\ell(n)}V=V,
\end{equation}
and \begin{equation} \label{eq01}
\lim_{n\to \infty} \Big(2\cdot 3^{\ell(n)}\Big)^{\frac 2 {\gb}}V_{2\cdot 3^{\ell(n)}}(P)=\lim_{n\to \infty} 2^{\frac 2 \gb}\frac 1{r^{2\ell(n)}} \frac 1 {3^{\ell(n)}} r^{2\ell(n)} 3^{\ell(n)}V(P; \gb_2)=2^{\frac 2 \gb} V(P; \gb_2).
\end{equation}
By \eqref{eq00} and \eqref{eq01}, we see that $\Big\{n^{\frac 2 {\gb}}V_{n}(P)\Big\}$ has two different subsequences having two different limits, and so $\lim_{n\to \infty} n^{\frac 2 \gb} V_n(P)$ does not exist. Due to Theorem~\ref{Th3}, and Proposition~\ref{prop004}, similarly, we can show that if $r_0\leq r\leq r_1$, then $\lim_{n\to \infty} n^{\frac 2 \gb} V_n(P)$ does not exist. Thus, we show that the $\gb$-dimensional quantization coefficient for $0<r\leq r_1$ does not exist, which completes the proof of the theorem.
\end{proof}

\subsection*{Acknowledgement} The author would like to express his sincere gratitude to the referees for their valuable comments.


\begin{thebibliography}{9999}
%

\bibitem{DFG} Du, Q., Faber, V., Gunzburger, M.: \emph{Centroidal Voronoi Tessellations: Applications and Algorithms}. Siam Rev., 41(4), 637-676 (1999)


\bibitem{DR} Dettmann, C.P. and Roychowdhury, M.K.: \emph{Quantization for uniform distributions on equilateral triangles}. Real Analysis Exchange, 42(1), 149-166 (2017)


\bibitem{GG} Gersho, A. and Gray, R.M.: \emph{Vector quantization and signal compression}. Kluwer Academy publishers, Boston (1992)




\bibitem{GL1} Graf, S. and Luschgy, H.: \emph{Foundations of quantization for probability distributions}. Lecture Notes in Mathematics 1730, Springer, Berlin (2000)


\bibitem{GL2} Graf, S. and Luschgy, H.: \emph{The quantization of the Cantor distribution}. Math. Nachr., 183, 113-133 (1997)

\bibitem {GL3} Graf, S. and Luschgy, H.: \emph{Asymptotics of the quantization error for self-similar probabilities}. Real Anal. Exchange,  26, 795-810 (2001)

\bibitem{GL4} Graf, S. and Luschgy, H.: \emph{Quantization for probability measures with respect to the geometric mean error}. Math. Proc. Camb. Phil. Soc., 136, 687-717 (2004)

\bibitem {GL5} Graf, S., Luschgy, H. and Pag\'es, G.: \emph{The local quantization behavior of absolutely continuous probabilities}.  Ann. Probab. 40(4), 1795-1828 (2012)

\bibitem{GN}  Gray, R. and Neuhoff, D.: \emph{Quantization}. IEEE Trans. Inform. Theory,  44,  2325-2383 (1998)


\bibitem{H} Hutchinson, J.: \emph{Fractals and self-similarity}. Indiana Univ. J., 30, 713-747 (1981)


\bibitem{P}  P\"otzelberger, K.: \emph{The quantization dimension of distributions}. Math. Proc. Camb. Phil. Soc., 131, 507-519 (2001)

\bibitem{P1} Pag\`es, G.: \emph{A space quantization method for numerical integration}. J. Comput. Appl. Math., 89, 1-38 (1998)
\bibitem{P2} Pag\`es, G.  and Printems, J.:  \emph{Functional quantization for numerics with an application to option pricing}. Monte Carlo Methods
Appl., 11, 407-446 (2005)

%
%
%
%
%

\bibitem{R1} Roychowdhury, M.K.: \emph{Least upper bound of the exact formula for optimal quantization of some uniform Cantor distributions}. Discrete and Continuous Dynamical Systems- Series A, 38(9), 4555-4570 (2018)

\bibitem{R2} Roychowdhury, M.K.:  \emph{The quantization of the standard triadic Cantor distribution}. Houston Journal of Mathematics, 46(2),  389-407 (2020)
\bibitem{RR} Rosenblatt, J. and Roychowdhury, M.K.: \emph{Optimal quantization for piecewise uniform distributions}.  Uniform Distribution Theory, 13(2), 23-55 (2018)


 \bibitem{Z} Zador, P.L.: \emph{Development and Evaluation of Procedures for Quantizing Multivariate Distributions}. Ph.D. Thesis (Stanford
University, (1964)



\end{thebibliography}
\end{document}